\documentclass{rse}
\usepackage{amsmath, amsthm}
\title{ VECTOR FIELDS ON RIGHT GENERALIZED COMPLEX PROJECTIVE STIEFEL MANIDFOLDS.}
\author{{Shilpa Gondhali, B. Subhash}}
\newcommand{\ncom}{\newcommand}
\newtheorem{thm}{Theorem}[section]

\ncom{\eop}{{\hfill$\spadesuit$}}

\begin{document}
\maketitle
\begin{abstract}
The question of paralleizability and stable parallelizability of a family of manifolds obtained as a quotients of circle
action on the complex Stiefel manifolds are studied and settled.
\end{abstract}
\section{Introduction}
Consider $\mathbb{C}^n$ equipped with the standard Hermitian product. The space of all orthonormal 
$k$-frames in $\mathbb{C}^n$ is called the complex Stiefel manifolds and is denoted by $W_{n,k}$. We can identify
$W_{n,k}$ with the homogeneous space $U(n)/ U(n-k)$ where $U(n)$ denote the group of unitary transformations 
of $\mathbb{C}^n$ and $U(n-k)$ is imbedded in $U(n)$ as the subgroup that fixes the first $k$ standard basis vectors
$e_1, \ldots, e_k \in \mathbb{C}^n$.\\

\noindent
Let $S^1$ denote the circle group, $\{z\in \mathbb{C} \mid |z| =1\}.$ Consider the action of $S^1$ 
on $W_{n,k}$ given by $z \cdot (v_1,\ldots,v_k)=(z\cdot v_1,\ldots,z \cdot v_k),$ where $(v_1, \ldots, v_k) \in W_{n,k}$ 
and $z \in S^1 .$ The quotient space of $W_{n,k}$ under this action of $S^1$ is called projective Stiefel manifold, 
usually denoted by 
$PW_{n,k}.$ The projective Stiefel manifold, $PW_{n,k}$ can also be realised as the homogeneous space 
$U(n)/(S^1 \times U(n-k))$ and is well studied [\ref{agmp1}], [\ref{agmp2}].\\

\noindent
This action of $S^1$ on $W_{n,k}$ can be generalised as follows: Let $(v_1,\ldots,v_k) \in W_{n,k}$ and $z \in S^1$. 
For any $n$-tuple of integers $(p_1,\ldots,p_n)$, there is an $S^1$ action on $W_{n,k}$ given by 
$z \cdot (v_1,\ldots, v_k) =  (\text{diag}(z^{p_1},\ldots,z^{p_n})v_1,\ldots, \text{diag}(z^{p_1},\ldots,z^{p_n})v_k)$ 
where $\text{diag}(z^{p_1},\ldots,z^{p_n})$ is the diagonal matrix with entries $z^{p_1},\ldots, z^{p_n}$. 
This action is free if and only if every collection of $k$ of these integers is relatively prime. 
In that case the quotient space is a manifold. It is known as the generalized projective Stiefel manifold and is 
denoted by $M_k(p_1,\ldots, p_n).$ Notice that $M_k(p_1,\ldots ,p_n)$ can be identified with
 $S^1\backslash U(n)/U(n-k)$ and it is well known that the space is non-homogeneous in general (For detailed discussion
one can read [\ref{bgm}]). The question of paralleizability for these manifolds 
has been studied and settled in [\ref{am03}].\\
 
\noindent
We consider another generalisation of the action of $S^1$ on $W_{n,k}$ which is described as follows: 
$$ z\cdot (v_1,v_2, \ldots, v_k) = (z^{l_1}v_1,z^{l_2}v_2, \ldots,z^{l_k}v_k)$$ 
for $z \in S^1$ and $l_1, l_2, \ldots, l_k \in \mathbb{Z}$. Note that this action is the same as multiplication from the 
right by the $ k \times k $ diagonal matrix with diagonal entries $z^{l_1},z^{l_2}, \ldots ,z^{l_k}.$
Given ${\bf {\ell}} := (l_1,l_2, \ldots, l_k),$ let $W_{n,k;\bf{\ell}}$ 
denote the orbit space $W_{n,k}/S^1$. We can realise $W_{n,k;\bf{\ell}}$ as the homogeneous space $U(n)/H\times U(n-k)$,
where $H$ is the group $S^1$ embedded in $U(k)$ as $\mbox{diag}(z^{l_1},z^{l_2},\ldots, z^{l_k}), ~~z \in S^1.$
Observe that if $\text{g.c.d.}(l_1,l_2,\ldots,l_k) =1,$ then $W_{n,k;\bf{\ell}}$ is a manifold and will be referred to as 
{\it{right generalized complex projective Stiefel manifold}}. Note that when 
$l_i =1 ~\forall ~ i ,~~W_{n,k;\bf{\ell}}$ is the complex projective Stiefel manifold.
\noindent
In this paper we study the paralleizability and the span of this manifold $W_{n,k;\bf{\ell}}$. 
Throughout the paper $\cong$ is used to denote isomorphism of real bundles.
In \S \ref{tangentbundle}, we obtain an expression for the tangent bundle of $W_{n,k;\bf{\ell}}$. 
In \S \ref{span}, 
we settle the question of paralleizability and stable paralleizability with the following theorem

\begin{thm}
 Let $n\geq 2$, if $k= n,~~ n-1,$ then the manifolds $W_{n,k;\bf{\ell}}$ is parallelizable, 
except $W_{2,1;\bf{\ell}}$ which is stably parallelizable. The manifolds $W_{n,k;\bf{\ell}} $ is not 
stably parallelizable otherwise.
\end{thm}

%%%%%%%%%%%%%%%%%%%%%%%%%%%%%%%%%%%%%%%%%%%%%%%%%%%%%%%%%%%%%%%%%%%%%%%%%%%%%%%%%%%%%%%%%%%%%%%%%%%%%%%%%%%%%%%%%%%%%%%%%%
\section{Tangent Bundle}\label{tangentbundle}
\noindent
We know that $\pi _{\bf{\ell}} :W_{n,k} \longrightarrow W_{n,k;\bf{\ell}}$ is a principal fibre bundle with fibre 
and structure group $S^1.$ 
Let $\xi_{n,k;\bf{\ell}}$ denote the associated complex line bundle. 
We have a natural projection map 
\[
 q : W_{n,k;\bf{\ell}} \longrightarrow Flag_{\mathbb{C}}(1,\ldots,1, n-k)
\] 
defined by 
$q([v_1,v_2,\ldots, v_k]) = (\mathbb{C} v_1,\mathbb{C} v_2, \ldots, \mathbb{C} v_k,\langle v_1, v_2 ,\ldots, v_k \rangle ^\perp)$ 
where
the orthogonal complement is taken with respect to the standard Hermitian product on 
$\mathbb{C}^n$. 
The tangent bundle of the complex flag manifold $Flag_{\mathbb{C}} (n_1,n_2,\ldots, n_s) $ %\cite{l75} 
is given by
\[
\tau Flag_{\mathbb{C}}(n_1,n_2,\ldots,n_s) \cong_{\mathbb{C}} \displaystyle{\bigoplus_{1\leq i < j\leq s }\xi_i^\vee \otimes_{\mathbb{C}} \xi_j},
\]
where $\cong_{\mathbb{C}}$ denotes isomorphism as real vector bundles (See [\ref{l75}] for details). 
Here $\xi_i $ is the complex vector bundle over the complex flag manifold, 
$Flag_{\mathbb{C}}(n_1,n_2,\ldots,n_s)$ whose 
fibre at a point $V = (V_1,V_2,\ldots,V_s)$ is $V_i$ for all $i = 1,2, \ldots, s$.
Hence we have 
\begin{equation}\label{pull-back}
\tau W_{n,k;\bf{\ell}} \cong q^* (\tau Flag_{\mathbb{C}} (1,1,\ldots,1,n-k)) \oplus \nu ,
\end{equation}
where $\nu$ denotes the real vector bundle tangential to the 
fibres of $q.$ Since, $q$ is a principal bundle with fibre and structure group $(S^1)^k/S^1,$ 
we have $\nu$ is trivial of rank $k-1.$ It follows that
\begin{equation}
 \tau W_{n,k;\bf{\ell}} \cong q^* (\tau Flag_{\mathbb{C}} (1,1,\ldots,1,n-k))\oplus (k-1) \varepsilon _{\mathbb{R}}.
\end{equation}
Observe that $q^*(\xi _i) \cong \xi_{n,k;\bf{\ell}}^{l_i},$ for $1 \leq i \leq k.$ 
Denote $\beta_{n,k;\bf{\ell}} := q^*(\xi_{k+1}),$ we have,
\begin{eqnarray*}
 \tau W_{n,k;\bf{\ell}} & \cong & q^*\big(\displaystyle{\bigoplus_{1\leq i < j \leq k+1} \xi_i^\vee 
                                    \otimes_{\mathbb{C}} \xi_j\big)} \oplus (k-1) 
                                    \varepsilon _{\mathbb{R}},\\
                                 & \cong &  \displaystyle{\bigoplus_{1\leq i< j \leq k+1}} q^*(\xi_i^\vee 
                                   \otimes_{\mathbb{C}} \xi_j) \oplus (k-1) 
                                    \varepsilon _{\mathbb{R}},\\
                                & \cong & \displaystyle{\bigoplus_{1\leq i < j \leq k}} 
                                   \xi_{n,k;\bf{\ell}}^{-l_i} \otimes_{\mathbb{C}} \xi_{n,k;\bf{\ell}}^{l_j}  
                                   \oplus \displaystyle{\bigoplus_{1\leq i \leq k}}\xi_{n,k;\bf{\ell}}^{-l_i}\otimes_{\mathbb{C}}
                                   \beta_{n,k;\bf{\ell}}\oplus (k-1) \varepsilon _{\mathbb{R}}.
\end{eqnarray*}

\noindent
As an immediate corollary we get :

\begin{lemma}\label{tangent}
 The tangent bundle $\tau W_{n,k;\bf{\ell}}$ of $W_{n,k;\bf{\ell}}$ satisfies the following equation 
\[ 
 \tau W_{n,k;\bf{\ell}} \oplus (k+1) \varepsilon _{\mathbb{R}} \oplus \displaystyle{\bigoplus_{1 \leq j < i \leq k}} 
\xi_{n,k;\bf{\ell}}^{-l_i} \otimes_{\mathbb{C}} \xi_{n,k;\bf{\ell}} ^{l_j} \cong  n 
\displaystyle{\bigoplus_{1 \leq i \leq k} \xi_{n,k;\bf{\ell}}^{- l_i}}.
\]
\end{lemma}
\begin{proof} 
We know that
\[
\xi_1 \oplus \xi_2 \oplus \ldots \oplus \xi_{k+1} \cong_{\mathbb{C}} n \varepsilon _{\mathbb{C}}.
\]
Multiply the above equation by $ \xi_i^ \vee$ where $ 1 \leq i \leq k$ to get 
\[
\displaystyle{ \bigoplus _{1\leq j \leq k+1,~j \neq i} } \xi_i^\vee \otimes_{\mathbb{C}} \xi_j 
\oplus \varepsilon_{\mathbb{C}} \cong_{\mathbb{C}} n \xi_i^\vee \mbox{ for } i =1,2,\ldots, k.
\]
Taking the sum over all $i = 1,2, \ldots, k$ and taking pull back under $q^*$ we get, 
\[
 \tau W_{n,k;\bf{\ell}} \oplus (k+1) \varepsilon _{\mathbb{R}} \oplus \displaystyle{ \bigoplus _{1\leq j < i \leq k}} 
q^*(\xi_i^\vee \otimes_{\mathbb{C}} \xi_j) \cong  n q^*(\displaystyle { \bigoplus_{1\leq i \leq k }} \xi_i^\vee).
\]
Hence we get 
\[
\tau W_{n,k;\bf{\ell}} \oplus (k+1) \varepsilon _{\mathbb{R}} \oplus \displaystyle{ \bigoplus _{1\leq j < i \leq k}}
\xi_{n,k;\bf{\ell}}^{-l_i} \otimes_{\mathbb{C}} \xi_{n,k;\bf{\ell}} ^{l_j} \cong  
n\displaystyle{ \bigoplus _{ 1 \leq i \leq k}} \xi_{n,k;\bf{\ell}}^{-l_i}.
\]
\end{proof}
%%%%%%%%%%%%%%%%%%%%%%%%%%%%%%%%%%%%%%%%%%%%%%%%%%%%%%%%%%%%%%%%%%%%%%%%%%%%%%%%%%%%%%%%%%%%%%%%%%%%%%%%%%%%%%%%%%
\section{Span and Stable span}\label{span}
\noindent
Recall that if $\omega _1,~\omega _2$ are any two complex vector bundles over the same base space $B$ and if 
\[ 
c(\omega _i ) = \displaystyle{\prod _{j=1}^ {n_i}} (1 + x_j^i)
\] 
are the formal decomposition of their Chern polynomials, then 
\[
 c(\omega_1 \otimes \omega _2) = \displaystyle{\prod _{i=1}^{n_1}} \displaystyle{\prod _{j=1}^{n_2}}(1+x_i^1+x_j^2).
\]
See \cite{bh58} for details. We also know that the Chern classes satisfy Whitney product formula, i.e., 
$c(\omega _1\oplus \omega _2) = c(\omega _1) \smile c(\omega_2)$ where $\omega _1,~ \omega _2$ are as mentioned above 
and $\smile$ denotes the cup product. We will write $\cdot$ instead of $\smile$ when there is no risk of confusion.\\

\noindent
Also, if $\xi$ is a complex $n$ plane bundle, the Chern classes $c_i$ of $\xi$ and the Pontrjagin classes $p_k$ of 
the underlying real bundle of $\xi$ are related (p. 177, cor. 15.5 of [\ref{ms74}]) by the following formula 
\begin{equation}\label{rel}
 1-p_1+p_2-+\cdots\pm p_n = (1-c_1+c_2 -+ \cdots \pm c_n)(1+c_1+c_2 +\cdots +c_n).
\end{equation}

\begin{lemma}\label{chclass}
 For $k<n-1,$ the complex line bundle $\xi_{n,k;\bf{\ell}}$ is not trivial.
\end{lemma}
\begin{proof}
 We have a principal fibre bundle $\pi _{\ell} :W_{n,k} \longrightarrow W_{n,k;\bf{\ell}}$ with fibre and structure group 
$S^1,$ denoted by $\xi_{n,k;\bf{\ell}}.$ Let $c_1 := c_1(\xi_{n,k;\bf{\ell}}) $ denote the first Chern class of this 
bundle. Considering the associated Gysin sequence, we have
$0\rightarrow H^1( W_{n,k})\rightarrow H^0( W_{n,k;\bf{\ell}})\displaystyle{\xrightarrow{c_1}} H^2( W_{n,k;\bf{\ell}}) 
\rightarrow H^2(W_{n,k})\rightarrow H^1( W_{n,k;\bf{\ell}}) \displaystyle{\xrightarrow {c_1}} H^3( W_{n,k;\bf{\ell}})
\rightarrow  H^3(W_{n,k}) \rightarrow H^2( W_{n,k;\bf{\ell}}) \displaystyle{\xrightarrow {c_1}} H^4( W_{n,k;\bf{\ell}}) 
\rightarrow H^4(W_{n,k}) \rightarrow \cdots  $

\noindent
is a long exact sequence. For $k<n-1,$ $W_{n,k}$ is 4 - connected implies $H^1( W_{n,k}) = H^2( W_{n,k}) = H^3( W_{n,k}) = H^4( W_{n,k}) = 0 .$ Hence  it follows that $H^0( W_{n,k;\bf{\ell}})\cong H^2( W_{n,k;\bf{\ell}})\cong H^4( W_{n,k;\bf{\ell}})\cong \mathbb{Z}$ and $H^2( W_{n,k;\bf{\ell}}),~~H^4( W_{n,k;\bf{\ell}})$ are generated by $c_1,~~c_1^2$ respectively. Hence $c_1(\xi_{n,k;\bf{\ell}}) \neq 0 \neq c_1(\xi_{n,k;\bf{\ell}})^2$ which imply that $\xi_{n,k;\bf{\ell}}$ is not a trivial bundle.
\end{proof}

\begin{thm}
 Let $n\geq 2$, if $k= n,~~ n-1,$ then the manifold $W_{n,k;\bf{\ell}}$ is parallelizable, except $W_{2,1;\bf{\ell}}$ which is stably parallelizable. 
The manifold $W_{n,k;\bf{\ell}} $ is not stably parallelizable otherwise.
\end{thm}
\begin{proof}
 In \S \ref{tangentbundle}, we have seen that 
\[
\tau W_{n,k;\bf{\ell}} \cong q^* (\tau Flag_{\mathbb{C}} (1,1,\ldots,1,n-k))\oplus (k-1) \varepsilon _{\mathbb{R}}.
\]
Since $Flag_{\mathbb{C}} (1,1,\ldots,1,1)$ is stably parallelizable, it follows that $W_{n,k;\bf{\ell}}$ 
is parallelizable when $k= n, ~~n-1,$ except $W_{2,1;\bf{\ell}}$ which is stably parallelizable\\
When $k< n-1,$ by lemma (\ref{tangent}) we have the first Pontrjagin class, 
\[
 p_1(\tau W_{n,k;\bf{\ell}} \oplus (k+1) \varepsilon _{\mathbb{R}} \oplus 
\displaystyle{\bigoplus_{1 \leq j < i \leq k}} \xi_{n,k;\bf{\ell}}^{-l_i} 
\otimes_{\mathbb{C}} \xi_{n,k;\bf{\ell}}^{l_j}) =  p_1( n \displaystyle{\bigoplus_{1 \leq i \leq k} 
\xi_{n,k;\bf{\ell}}^{- l_i}}).
\]
Hence we get,
\[
p_1(\tau W_{n,k;\bf{\ell}}) + \displaystyle{\sum_{1 \leq j < i \leq k}}
p_1(\xi_{n,k;\bf{\ell}}^{-l_i} 
\otimes_{\mathbb{C}} \xi_{n,k;\bf{\ell}}^{l_j}) =  n \displaystyle{\sum_{1 \leq i \leq k} 
p_1(\xi_{n,k;\bf{\ell}}^{- l_i}}).
\]
The relation between Chern and Pontrjagin classes (\ref{rel})  for line bundle implies
\[
 p_1(\tau W_{n,k;\bf{\ell}}) + \displaystyle{\sum_{1 \leq j < i \leq k}}
c_1(\xi_{n,k;\bf{\ell}}^{-l_i} 
\otimes_{\mathbb{C}} \xi_{n,k;\bf{\ell}}^{l_j} )^2 =  n \displaystyle{\sum_{1 \leq i \leq k}  
c_1(\xi_{n,k;\bf{\ell}}^{- l_i}})^2.
\]
Hence, we obtain 
\[
p_1(\tau W_{n,k;\bf{\ell}}) + \displaystyle{\sum_{1 \leq j < i \leq k}}
{(l_j-l_i)^2c_1(\xi_{n,k;\bf{\ell}})^2} 
= n \displaystyle{\sum_{1 \leq i \leq k}{ l_i}^2 c_1(\xi_{n,k;\bf{\ell}})^2}.
\]
This implies
\begin{eqnarray*}
p_1(\tau W_{n,k;\bf{\ell}})& =&  \Big( n \displaystyle{\sum_{1 \leq i \leq k}{ l_i}^2} -
\displaystyle{\sum_{1 \leq j < i \leq k}}{(l_j-l_i)^2}\Big) c_1(\xi_{n,k;\bf{\ell}})^2, \\
&=&\Big( (n-k+1)\displaystyle{\sum_{1 \leq i \leq k}{ l_i}^2} +
\displaystyle{\sum_{1 \leq j < i \leq k}}{2l_jl_i}\Big) c_1(\xi_{n,k;\bf{\ell}})^2 ,\\
&=&\Big( (n-k)\displaystyle{\sum_{1 \leq i \leq k}{ l_i}^2} +(\displaystyle{\sum_{1 \leq i \leq k}{ l_i}})^2\Big) c_1(\xi_{n,k;\bf{\ell}})^2 \neq 0.\\
\end{eqnarray*}
If $k< n-1,$ we have seen that $c_1(\xi_{n,k;\bf{\ell}})^2 \neq 0$ in proof of lemma (\ref{chclass}). This
shows that the manifolds $W_{n,k;\bf{\ell}}$ is not stably parallelizable if $k< n-1.$
\end{proof}

\noindent Notice that
$W_{n,k;{\bf{\ell}}}$ is orientable since it is a quotient of $SU(n)$ by a connected closed subgroup.
In view of (\ref{pull-back}) and the fact that $\nu$ is trivial, all odd Stiefel- Whitney classes of
$W_{n,k;{\bf{\ell}}}$ are zero.
Moreover, the Whitney product formula implies,
\[ 
w(\tau W_{n,k;\bf{\ell}})w( \displaystyle{\bigoplus_{1 \leq j < i \leq k}} 
\xi_{n,k;\bf{\ell}}^{-l_i} \otimes_{\mathbb{C}} \xi_{n,k;\bf{\ell}} ^{l_j})= w(\displaystyle{
\bigoplus_{1 \leq i \leq k} \xi_{n,k;\bf{\ell}}^{- l_i}})^n.
\]
Hence,
$w_2(\tau W_{n,k;\bf{\ell}}) + w_2(\displaystyle{\bigoplus_{1 \leq j < i \leq k}} 
\xi_{n,k;\bf{\ell}}^{-l_i} \otimes_{\mathbb{C}} \xi_{n,k;\bf{\ell}} ^{l_j}) = w_2(\displaystyle{
\bigoplus_{1 \leq i \leq k} \xi_{n,k;\bf{\ell}}^{- l_i}})$
\begin{eqnarray*}
w_2(\tau W_{n,k;\bf{\ell}})& = &w_2(\displaystyle{
\bigoplus_{1 \leq i \leq k} \xi_{n,k;\bf{\ell}}^{- l_i}}) - w_2(\displaystyle{\bigoplus_{1 \leq j < i \leq k}} 
\xi_{n,k;\bf{\ell}}^{-l_i} \otimes_{\mathbb{C}} \xi_{n,k;\bf{\ell}} ^{l_j})\\
&=& \{c_1(\displaystyle{
\bigoplus_{1 \leq i \leq k} \xi_{n,k;\bf{\ell}}^{- l_i}}) - c_1(\displaystyle{\bigoplus_{1 \leq j < i \leq k}} 
\xi_{n,k;\bf{\ell}}^{-l_i} \otimes_{\mathbb{C}} \xi_{n,k;\bf{\ell}} ^{l_j})\}\pmod{2} \\
&=& \{[n \displaystyle{\sum_i l_i - \sum_{j<i} (l_j-l_i)}] c_1(\xi_{n,k;\bf{\ell}})\}\pmod{2}\\
 w_2(\tau W_{n,k;\bf{\ell}})& = &[n \displaystyle{\sum_i l_i - \sum_{j<i} (l_j-l_i)}] \pmod{2}~ w_2(\xi_{n,k;\bf{\ell}}).
\end{eqnarray*}

\noindent
If $r$ denote the number of $l_1, \cdots,l_k$ that are even, then 
$$w_2(\tau W_{n,k;\bf{\ell}}) = (n+r)(k-r)w_2(\xi_{n,k;\bf{\ell}}).$$ Note that $w_2 \neq 0$ only if $n,~k
 \mbox{ are odd and } r \mbox{ even or } n,~k \mbox{ are even and } r \mbox{ odd}.$ 
This fact along with the theorem from [\ref{k81}] implies the following
\begin{thm}
$\mbox{Span}(W_{n,k;\bf{\ell}}) = \mbox{ Stable Span}(W_{n,k;\bf{\ell}}),$ 
 in each of the following cases  
\begin{enumerate}
  \item $k>1$ is odd,
  \item $k \equiv 2 \pmod{4},~~ k>2$ and $n$ odd,
  \item $k \equiv 2 \pmod{4},~~ k>2, ~~n$ even and an even number of $l_1, \cdots,l_k$ are even.
 \end{enumerate}

\noindent
{\bf Acknowledgment:} We thank Parameswaran Sankaran for suggesting problem and for all discussions.
 We thank Anant Shastri for a careful reading of the paper and for his valuable comments.
\end{thm}

\refname{}

\begin{enumerate}
  
\item\label{agmp1} L. Astey, S. Gitler, E. Micha, G. Pastor.; Cohomology of complex projective Stiefel Manifolds. 
             Canad. J. Math. {\bf 51} (1999), 897 -- 914. 
\item\label{agmp2} L. Astey, S. Gitler, E. Micha, G. Pastor.; Parallelizability of complex projective Stiefel Manifolds. 
             Proc. Amer. Math. Soc. {\bf vol 128}, {\bf no 5} (1999), 1527 -- 1530.

\item\label{am03} L. Astey, E. Micha.; On the paralleizability of generalised complex projective Stiefel manifolds. 
            Proc. of royal society of Edinburgh {\bf 133 A} (2003), 497 -- 504.

\item\label{bgm} C. Boyer, K. Galicki, B. Mann; New examples of inhomogeneous Einstein manifolds of positive 
           sectional curvature. Math. Res. Lett {\bf 1} (1994), 115 -- 121.

\item\label{bh58} Borel, A., Hirzebruch, F.; Characteristic classes and homogeneous spaces. I. 
                       Amer. J. Math. {\bf 80} (1958), 458 -- 538.

\item\label{l75} Lam, Kee Yuen; A formula for the tangent bundle of flag manifolds and related 
                     manifolds. Trans. Amer. Math. Soc. {\bf 213} (1975), 305 -- 314.

\item\label{ms74} Milnor, John W., Stasheff, James D.; {\it{Characteristic classes}}. Annals of Mathematics 
              Studies, No. {\bf 76}. Princeton University Press, Princeton, N. J.; University of Tokyo 
              Press, Tokyo, 1974.
\item\label{k81} Koschorke, Ulrich.; {\it{Vector fields and other vector bundle morphisma- A singularity approach}}. 
              Lecture Notes in Mathematics, {\bf 847}, Springer, Berlin, 1981. 

\end{enumerate}
\begin{description}
 \item[] Shilpa Gondhali \\ Tata institute of Fundamental Research, Mumbai. \\ email : sshilpaa@gmail.com, shilpa@math.tifr.res.in
 \item[] B. Subhash \\ Institute of Mathematical Sciences, Chennai.\\ email : subhash02@gmail.com, subi@imsc.res.in
\end{description}
\end{document}